\documentclass[11pt,a4paper]{article}

\usepackage{amsmath}
\usepackage{amssymb,latexsym}
\usepackage{mathabx}
\usepackage{amsthm}
\usepackage[shortlabels]{enumitem}
\usepackage{array}
\usepackage[hang,flushmargin,bottom]{footmisc}
\usepackage[margin=0.35in,format=hang,font=small,labelfont=bf]{caption}
\usepackage[normalem]{ulem}
\usepackage{needspace}
\usepackage{graphicx,tikz}
\usepackage[colorlinks=true,urlcolor=blue,linkcolor=blue,citecolor=blue]{hyperref}
\usepackage{palatino}
\usepackage{eulervm}

\usepackage{ifthen}
\usetikzlibrary{calc}
\usetikzlibrary{decorations.pathreplacing,decorations.markings}

\newcommand{\plotptradius}{4pt}

\tikzset{permpt/.style={circle, draw, fill=black, inner sep=0pt, minimum width=\plotptradius}}
\tikzset{empty/.style={draw=none, fill=none}}

\newcommand\absdot[2]{
	\node[permpt] at #1 {};
}

\newcommand{\plotperm}[2][black]{ 
	\foreach \j [count=\i] in {#2} {
        \ifnum0=\j {} \else {
 		\node[permpt,fill=#1,draw=#1] (\j) at (\i,\j) {};
	} \fi
	};
}
\newcommand{\plotpermbox}[4]{
	\draw [darkgray, very thick, line cap=round, fill=white]
		({#1-0.5}, {#2-0.5}) rectangle ({#3+0.5}, {#4+0.5});
}

\newcommand{\plotpermborder}[1]{
	\foreach \i [count=\nn] in {#1} {\global\let\n\nn};    
	\plotpermbox{1}{1}{\n}{\n};
	\plotperm{#1};
}
\newcommand{\plotpermbordergrid}[1]{
	\foreach \i [count=\nn] in {#1} {\global\let\n\nn};    
	\plotpermbox{1}{1}{\n}{\n};
	\draw[step=1cm,gray!50,very thin] (1,1) grid (\n,\n);
	\plotperm{#1};
}
\newcommand{\plotgrid}[1]{
	\draw[step=1cm,gray!50,very thin] (0.5,0.5) grid (#1+0.5,#1+0.5);
}

\newcommand{\plotpermgrid}[1]{
	\foreach \i [count=\nn] in {#1} {\global\let\n\nn};    
	\plotgrid{\n};
	\plotperm{#1};
}

\newcommand{\plotpinsequence}[1]{
	\absdot{(0,0)}{};
	\edef\n{0}
	\edef\s{0}
	\edef\e{0}
	\edef\w{0}
	\edef\x{0}
	\edef\y{0}
	\foreach \pin [remember=\pin as \oldpin (initially 1), count=\i] in {#1} {
		\ifthenelse{\pin=1 \OR \pin=2}{
			\ifthenelse{\oldpin=3}{
				\xdef\x{\number\numexpr\e-1}
			}{
				\xdef\x{\number\numexpr\w+1}
			}
			\ifnum\i=1 
				\pgfmathparse{\e+1}
 				\xdef\e{\pgfmathresult}
			\fi	
		}{ 
			\ifthenelse{\oldpin=1}{
				\xdef\y{\number\numexpr\n-1}
			}{
				\xdef\y{\number\numexpr\s+1}
			}
			\ifnum\i=1 
				\pgfmathparse{\s-1}
 				\xdef\s{\pgfmathresult}
			\fi	
		}
		\ifnum\pin=1 
			\pgfmathparse{\n+2}
 			\xdef\n{\pgfmathresult}		
			\absdot{(\x,\n)}{};
			\ifnum\i>1
				\draw (\x,\n) -- (\x,\y-0.5);
			\else
				\draw[gray,very thick] (-0.5,-0.5) rectangle (\x+0.5,\n+0.5);
			\fi
		\fi
		\ifnum\pin=2 
			\pgfmathparse{\s-2}
 			\xdef\s{\pgfmathresult}
			\absdot{(\x,\s)}{};
			\ifnum\i>1
				\draw (\x,\s) -- (\x,\y+0.5);
			\else
				\draw[gray,very thick] (-0.5,0.5) rectangle (\x+0.5,\s-0.5);
			\fi
		\fi
		\ifnum\pin=3 
			\pgfmathparse{\e+2}
 			\xdef\e{\pgfmathresult}
			\absdot{(\e,\y)}{};
			\ifnum\i>1
				\draw (\e,\y) -- (\x-0.5,\y);
			\else
				\draw[gray,very thick] (-0.5,+0.5) rectangle (\e+0.5,\y-0.5);
			\fi
		\fi
		\ifnum\pin=4 
			\pgfmathparse{\w-2}
 			\xdef\w{\pgfmathresult}
			\absdot{(\w,\y)}{};
			\ifnum\i>1
				\draw (\w,\y) -- (\x+0.5,\y);
			\else
				\draw[gray,very thick] (0.5,0.5) rectangle (\w-0.5,\y-0.5);

			\fi
		\fi		
	};
}

\tikzset{
  on each segment/.style={
    decorate,
    decoration={
      show path construction,
      moveto code={},
      lineto code={
        \path [#1]
        (\tikzinputsegmentfirst) -- (\tikzinputsegmentlast);
      },
      curveto code={
        \path [#1] (\tikzinputsegmentfirst)
        .. controls
        (\tikzinputsegmentsupporta) and (\tikzinputsegmentsupportb)
        ..
        (\tikzinputsegmentlast);
      },
      closepath code={
        \path [#1]
        (\tikzinputsegmentfirst) -- (\tikzinputsegmentlast);
      },
    },
  },
  mid arrow/.style={postaction={decorate,decoration={
        markings,
        mark=at position .5 with {\arrow[#1]{stealth}}
      }}},
}
\usetikzlibrary{decorations,patterns}

\newtheorem{theorem}{Theorem}[section]
\newtheorem*{theorem*}{Theorem}

\newtheorem{lemma}[theorem]{Lemma}

\newtheorem*{conj*}{Conjecture}

\theoremstyle{definition}

\newtheorem{defn*}{Definition}

\newtheorem*{example*}{Example}

\newtheorem*{comment*}{Comment}

\let\C\CCC

\tikzstyle{vertex}=[circle, draw, fill=black,
                        inner sep=0pt, minimum width=4pt]

\title{Hereditary pattern-free classes are not always 2-wqo}

\author{Robert Brignall\\
\small School of Mathematics and Statistics\\
\small The Open University, UK
}

\begin{document}
\maketitle

\begin{abstract}
A pattern is a fundamental object used in the proof of Duron, M\"ahlmann and Toru\'nczyk to show that hereditary 2-wqo graph classes have bounded clique-width. We answer a question in that paper by exhibiting a pattern-free hereditary graph class that is not 2-wqo.
\end{abstract}

%
%
%
%
%
%
%
%
\section{Introduction}

Duron, M\"ahlmann, and Toru\'nczyk~\cite{duron:hereditary-2-wqo} recently proved that every 2-well-quasi-ordered (2-wqo) hereditary class has bounded clique-width and thus (via recent work of~Dumas and Lopez~\cite{dumas:wqo-clique-width}) gave a positive resolution to Pouzet's long-standing conjecture~\cite{pouzet:un-bel-ordre-da:}: 2-wqo is equivalent to labelled wqo. As part of this proof, they showed that every \(2\)-wqo class is pattern-free, and ask whether the converse holds. In this note we exhibit a class to show that this is not the case.

Let \(W_\infty\) be the (countably infinite) graph with vertex set $U\cup V$ where \(U=\{u_1,u_2,\ldots\}\) is an independent set,
\(V=\{v_1,v_2,\ldots\}\) is a clique, and $u_i v_j\in E(W_\infty)$ if and only if $i=j$ or $i\leq j-2$.
%
See Figure~\ref{fig:widdershins-graph} for an illustration of the first few vertices of $W_\infty$.  We consider the age of \(W_\infty\),
\[
    \mathcal C=\operatorname{Age}(W_\infty),
\]
that is, the hereditary class of all graphs isomorphic to finite
induced subgraphs of \(W_\infty\).

This class is precisely the class of graphs considered by Brignall, Engen and Vatter~\cite{bev:lwqo}, and there it is shown that $\C$ is not 2-wqo.\footnote{See Proposition 4.4 of~\cite{bev:lwqo}. Note that we do not define or use 2-wqo here.} Its construction derives from the study of permutation patterns -- in particular, $\C$ comprises all inversion graphs arising from the downward closure of the so-called `Widdershins spiral', which first appears in the PhD thesis of Max Murphy~\cite{murphy:restricted-perm:}. 

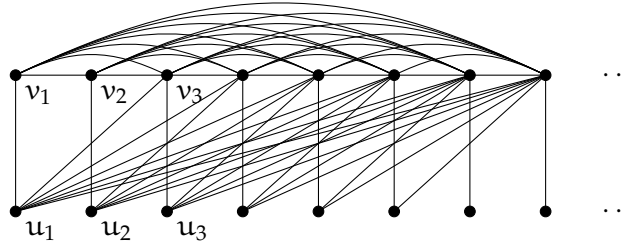
\begin{figure}[t]
    \centering
    \begin{tikzpicture}[xscale=1, yscale=1.8, baseline=(current bounding box.center)]
    \foreach \x/\y in {1/0,1/1,2/0,2/1,3/0,3/1,4/0,4/1,5/0,5/1,6/0,6/1,7/0,7/1,8/0,8/1}
     \node[permpt] at (\x,\y) {};
    \foreach \x in {1,2,3,4,5,6} {
      \draw (\x,0) -- (\x,1);
      \pgfmathsetmacro\ystart{\x + 2};
      \foreach \y in {\ystart,...,8}
      \draw (\x,0) -- (\y,1);
    }
    \draw (8,0) -- (8,1);
    \draw (7,0) -- (7,1);
    \draw (1,1) -- (8,1);
    \node at (1,1) [below right] {$v_1$};
    \node at (2,1) [below right] {$v_2$};
    \node at (3,1) [below right] {$v_3$};
    \node at (1,0) [below right] {$u_1$};
    \node at (2,0) [below right] {$u_2$};
    \node at (3,0) [below right] {$u_3$};
    \foreach \x/\y in {1/3,1/4,1/5,1/6,2/4,2/5,2/6,3/5,3/6,4/6,1/7,2/7,3/7,4/7,5/7,1/8,2/8,3/8,4/8,5/8,6/8}
      \draw (\x,1) to [in=165, out=15] (\y,1);	
      \node[empty] at (9,0) {$\cdots$};	
      \node[empty] at (9,1) {$\cdots$};	
  \end{tikzpicture}
    \caption{The first 16 vertices of graph \(W_\infty\).}
    \label{fig:widdershins-graph}
\end{figure}

%
%
%
%
%
%
%
%


For a positive integer $n$, let $[n]$ denote the set $\{1,\dots,n\}$. Let \(m,r\) be positive integers. Following~\cite{duron:hereditary-2-wqo}, an \emph{\((m,r)\)-pattern} is a
graph \(P\) with vertex set \([m]\times[r]\), and edges such that the following conditions are satisfied.
\begin{enumerate}[(i)]
    \item For each $j\in [r]$, the $j$th \emph{layer} \(L_j:=\{(i,j):i\in[m]\}\) is either a clique or an independent set.
    \item If \(|j-j'|>1\), then \(L_j\) and \(L_{j'}\) are either
          complete or anticomplete to one another.
    \item For each \(j\in[r-1]\), there is a relation $\triangleleft_j$ in $\{=,\neq,<,\leq,>,\geq\}$
          such that, for all \(i,i'\in[m]\),
          \[
              (i,j)(i',j+1)\in E(P)
              \quad\text{ if and only if }\quad
              i\mathrel{\triangleleft_j}i'.
          \]
\end{enumerate}
For convenience later, it is also helpful to define the $i$th \emph{row} by $R_i:=\{(i,j):j\in[r]\}$. We also remark here that these patterns are finite versions of the grids used in Brignall and Cocks' framework for unbounded clique-width~\cite{brignall:framework}.

A graph is \emph{\((m,r)\)-pattern-free} if it contains no
\((m,r)\)-pattern as an induced subgraph. A graph class is
\emph{\((m,r)\)-pattern-free} if each of its members is, and it is
\emph{pattern-free} if it is \((m,r)\)-pattern-free for some
positive integers \(m\) and \(r\).

The result in this note is the following.

\begin{theorem}\label{thm:main}
The class \(\C\) is \((14,4)\)-pattern-free.
\end{theorem}

%
%
%
%
%
%
%
%
\section{Proof of Theorem~\ref{thm:main}}
The class $\C$ contains a pattern $P$ if and only if $P$ is an induced subgraph of $W_\infty$. Thus, it suffices to consider patterns that are embedded in $W_\infty$.   Order the vertices of \(W_\infty\) by
\[
    v_1\prec u_1\prec v_2\prec u_2\prec\cdots.
\]
For \(x\prec y\), by definition we have
\begin{equation}\label{eq:order-adjacency}
    xy\in E(W_\infty)
    \quad\Longleftrightarrow\quad
    \begin{cases}
        y\in U, & \text{if \(x\) and \(y\) are consecutive in \(\prec\)},\\
        y\in V, & \text{otherwise}.
    \end{cases}
\end{equation}

For an interval \(I\) of $U\cup V$ under the \(\prec\) ordering, let \(I^+\) denote the interval
obtained by adjoining (when they exist) the vertices of $U\cup V$ that immediately
precede and succeed \(I\), and write \(I^{+t}\) for the result of
iterating this operation \(t\) times. It follows from~\eqref{eq:order-adjacency} that if \(x,x'\) are distinct vertices lying in an interval $I$ such that both belong to $U$ or both to $V$, then every vertex adjacent to exactly one of \(x,x'\) lies in \(I^+\).

We say that a pattern $P$ embedded in $W_\infty$ is \emph{clean} if each of its layers
is contained wholly in \(U\) or wholly in \(V\). 

\begin{lemma}\label{lem:interval-propagation}
Fix $m\geq 3$ and $r$. Let \(P\) be a clean \((m,r)\)-pattern in \(W_\infty\), and let \(L_j,L_{j+1}\) be consecutive layers of \(P\). Then the following hold.
\begin{enumerate}[(1)]
\item $L_j$ and $L_{j+1}$  lie in opposite parts.
\item The relation between $L_j$ and $L_{j+1}$ is one of $\{<,\leq,>,\geq\}$.
\item If an interval $I$ contains at least \(k\) vertices from \(L_j\), the interval \(I^+\) contains at least \(k-1\) vertices of \(L_{j+1}\).
\end{enumerate}
\end{lemma}

\begin{proof}
(1) If $L_j$ and $L_{j+1}$ were both contained in $U$, then the two layers would be anticomplete to one another. Similarly, if they were both contained in $V$, then they would be complete to one another. However, consecutive layers must be joined according to one of the six relations $ \{=,\neq,<,\leq,>,\geq\}$.

(2) 
From the definition of \(W_\infty\), whenever \(p+2\leq q\) we have
\[
    N(v_p)\cap U\subseteq N(v_q)\cap U,
    \qquad
    N(u_q)\cap V\subseteq N(u_p)\cap V.
\]
Now choose any three distinct vertices \(a\prec b\prec c\) in \(L_j\).
By the above, we have that one of \(N(a)\cap L_{j+1}\) and \(N(c)\cap L_{j+1}\) is a subset of the other. This rules out both the possibility that the relation between $L_j$ and $L_{j+1}$ is `$=$' (corresponding to a matching between the layers), and that the relation is `$\neq$' (corresponding to the bipartite complement of a matching).

(3) By (2), the neighbourhoods in \(L_{j+1}\) of the vertices of
\(L_j\) are distinct, and linearly ordered by inclusion. Choose \(k\) vertices of
\(L_j\cap I\), and among their neighbourhoods in \(L_{j+1}\) let
\(A\) and \(B\) be the smallest and largest. Clearly $A\subseteq B$, but we must also have $|B\setminus A|\geq k-1$, since there are at least $k-2$ nested sets lying between $A$ and $B$. 

Furthermore, every vertex of \(B\setminus A\) distinguishes the two corresponding vertices of \(L_j\). Since both of those vertices lie in \(I\), every
vertex of \(B\setminus A\) lies in \(I^+\). Hence
\[
    |L_{j+1}\cap I^+|\geq k-1,
\]
as required.
\end{proof}

\begin{proof}[Proof of Theorem~\ref{thm:main}]
Suppose for a contradiction that \(W_\infty\) contains a \((14,4)\)-pattern $P$. 
Since two vertices from $U$ are nonadjacent, a clique layer of $P$ contains at most one vertex of \(U\). Similarly, an independent layer contains at most one vertex of \(V\). Thus, for each layer delete the row containing its possible exceptional
vertex. After deleting at most four rows, choose ten of the remaining
rows and relabel them in increasing order. This gives a clean
\((10,4)\)-pattern. Denote the layers of this pattern by
\(L_1,L_2,L_3,L_4\).

Using the order $\prec$, list the vertices of \(L_1\) as $x_1\prec x_2\prec\cdots\prec x_{10}$, and let $I$ denote the smallest interval under $\prec$ containing all of $x_1,\dots,x_{10}$.
Applying Lemma~\ref{lem:interval-propagation}(3) three times to $I$ gives $|L_4\cap I^{+3}| \geq 7$.

By Lemma~\ref{lem:interval-propagation}(1), the layers of a clean pattern alternate between \(U\) and \(V\), so
\(L_1\) and \(L_4\) lie in opposite parts. Among the three vertices
added to each end of \(I\) in forming \(I^{+3}\), at most two lie
in the same part as \(L_4\). Consequently, at most four vertices of
\(L_4\cap I^{+3}\) lie outside \(I\). Within \(I\), the only possible
vertices of \(L_4\) that are consecutive to one of \(x_1,x_{10}\)
are the immediate successor of \(x_1\) and the immediate predecessor
of \(x_{10}\). Since \(|L_4\cap I^{+3}|\geq7\), there is therefore
some \(w\in L_4\) such that
\[
    x_1\prec w\prec x_{10},
\]
and \(w\) is consecutive to neither \(x_1\) nor \(x_{10}\).

It follows from \eqref{eq:order-adjacency} that
\[
    x_1w\in E(W_\infty)
    \quad\Longleftrightarrow\quad
    w\in V,
\]
and also
\[
    wx_{10}\in E(W_\infty)
    \quad\Longleftrightarrow\quad
    x_{10}\in V.
\]
However, $L_1$ and $L_4$ must lie in opposite parts, so exactly one of
these two pairs is an edge. Thus \(w\) has different adjacency to
\(x_1\) and \(x_{10}\), contradicting the requirement that the
nonconsecutive layers \(L_1\) and \(L_4\) are complete or anticomplete
to one another.

Therefore \(W_\infty\) contains no \((14,4)\)-pattern, and hence $\C$ is $(14,4)$-pattern-free.
\end{proof}

\bibliographystyle{plain}
\bibliography{refs}

\end{document}